\documentclass[12pt,a4paper]{amsart}

\newtheorem{theo+}              {Theorem}           [section]
\newtheorem{prop+}  [theo+]     {Proposition}
\newtheorem{coro+}  [theo+]     {Corollary}
\newtheorem{lemm+}  [theo+]     {Lemma}
\newtheorem{exam+}  [theo+]     {Example}
\newtheorem{rema+}  [theo+]     {Remark}
\newtheorem{defi+}  [theo+]     {Definition}

\newenvironment{theorem}{\begin{theo+}}{\end{theo+}}

\newenvironment{corollary}{\begin{coro+}}{\end{coro+}}

\usepackage{amsthm}
\theoremstyle{plain} \theoremstyle{remark}
\newtheorem{remark}{Remark}
\newtheorem{example}{Example}

\def\E{/\kern-1.0em \equiv }

\title{Linear $\infty$-Harmonic maps between Rienmannian manifolds}
\author{Ze-Ping Wang}

\address{Department of Mathematics $\&$ Physics
Yunnan Wenshan Teachers'College No. 2 Xuefu Road Wenshan County
Wenshan, Yunnan 653000 People's Republic of China.
\newline\indent E-mail:zeping.wang@gmail.com }

\begin{document}

\title[Linear $\infty$-harmonic maps]{Linear $\infty$-harmonic maps between Riemannian manifolds}

\subjclass{58E20, 53C12} \keywords{$\infty$-harmonic maps, Nil
space, Sol space, Heisenberg space.}

\maketitle

\section*{Abstract}
\begin{quote}
{\footnotesize In this paper, we give complete classifications of
linear $\infty$-harmonic maps between Euclidean and Heisenberg
spaces, between Nil and Sol spaces. We also classify all
$\infty$-harmonic linear endomorphisms of Sol space and show that
there is a subgroup of $\infty$-harmonic linear automorphisms in the
group of linear automorphisms of Sol space.}
\end{quote}
\section{introduction}
In this paper, all objects including manifolds, metrics, maps, and
vector fields are assumed to be smooth unless it is stated otherwise.\\

$\infty$-Harmonic functions are solutions of the so-called
$\infty$-Laplace equation:
\begin{equation}\notag
\Delta_{\infty}u:=\frac{1}{2}\langle\nabla\,u, \nabla\left|\nabla
u\right|^{2}\rangle=\sum_{i,j=1}^{m}u_{ij}u_{i}u_{j}=0,
\end{equation}
where $u:\Omega \subset \mathbb{R}^{m} \longrightarrow \mathbb{R}$,
$u_{i}=\frac{\partial u}{\partial x^{i}}$ and
$u_{ij}=\frac{\partial^{2} u}{\partial x^{i}\partial x^{j}}$. The
$\infty$-Laplace equation was first found by G. Aronsson
(\cite{Ar1}, \cite{Ar2}) in his study of ``optimal" Lipschitz
extension of functions in the late 1960s.

The $\infty$-Laplace equation can be obtained as the formal limit,
as $p\rightarrow \infty$, of $p$-Laplace equation
\begin{equation}
\Delta_{p}\,u:=\left|\nabla\,u\right|^{p-2}\left(\Delta\,u+
\frac{p-2}{\left|\nabla\,u\right|^{2}}\Delta_{\infty}\,u\right)=0.
\end{equation}

In recent years, there has been a growing research work in the study
of the $\infty$-Laplace equation. For more history and developments
see e.g. \cite{CIL}, \cite{ACJ}, \cite{BB}, \cite{Ba}, \cite{BLW1},
\cite{BLW2}, \cite{BEJ}, \cite{Bh}, \cite{CE}, \cite{CEG},
\cite{CY}, \cite{EG}, \cite{EY}, \cite{J}, \cite{JK}, \cite{JLM1},
\cite{JLM2}, \cite{LM1}, \cite{LM2}, \cite{Ob}. For interesting
applications of the $\infty$-Laplace equation in image processing
see \cite{CMS}, \cite{Sa}, in mass transfer problems see
e.g. \cite{EG}, and in the study of shape metamorphism see e.g. \cite{CEPB}.\\

Very recently, Ou, Troutman, and Wilhelm \cite{OW} introduced and
studied $\infty$-harmonic maps between Riemannian manifolds as a
natural generalization of $\infty$-harmonic functions and as a map
between Riemannian manifolds that satisfies a system of PDE
obtained as the formal limit, as $p\rightarrow \infty$, of
$p$-harmonic map equation:
\begin{equation}\notag
\frac{\left|{\rm d}\varphi\right|^{2}\tau_{2}\left(\varphi
\right)}{(p-2)}+\frac{1}{2}{\rm d}\varphi\left({\rm grad}\left|{\rm
d}\varphi\right|^{2}\right)=0.
\end{equation}
According to  \cite{OW}, a map $\varphi:(M,g)\longrightarrow (N,h)$
between Riemannian manifolds is called an  $\infty$-{\em harmonic
map} if the gradient of its energy density is in the kernel of its
tangent map, i.e., $\varphi$ is a solution of the $PDEs$
\begin{equation}\label{PL111}
\tau_{\infty}\left(\varphi \right)=\frac{1}{2}{\rm
d}\varphi\left({\rm grad}\left|{\rm d}\varphi\right|^{2}\right)=0.
\end{equation}
 where $\left|{\rm d}\varphi\right|^{2}=Trace_{g}\varphi^{*}h$ is the
 energy density of $\varphi$.

\begin{corollary}\label{PL124}(see \cite{OW})

In local coordinates, a map $\varphi:(M,g)\longrightarrow (N,h)$
 with\\ $\varphi(x)=(\varphi^{1}(x),\varphi^{2}(x),\ldots,\varphi^{n}(x))$
 is $\infty$-harmonic map if and only if
\begin{equation}\label{PL129}
g\left({\rm grad} \varphi^{i},{\rm grad}\left|{\rm
d}\varphi\right|^{2}\right)=0,\;\;\;\;\;\;\;i=1, 2,\ldots, n.
\end{equation}
\end{corollary}

\begin{example}(see \cite{OW})
Many important and familiar families of maps between Riemannian
manifolds turn out to be $\infty$-harmonic maps. In particular, all
maps of the following classes are $\infty$-harmonic:
\begin{itemize}
\item $\infty$-harmonic functions,
\item totally geodesic
maps,
\item isometric immersions,
\item Riemannian submersions,
\item eigenmaps between spheres,
\item projections of multiply warped products (e.g., the projection of
the generalized Kasner spacetimes),
\item equator maps, and
\item radial
projections.
\end{itemize}
\end{example} For more details of the above and other examples, methods
of constructing $\infty$-harmonic maps into Euclidean spaces and
into spheres, study of a subclass of $\infty$-harmonic maps called
$\infty$-harmonic morphisms, study of the conformal change of
$\infty$-Laplacian on Riemannian manifolds and other results we refer the readers to \cite{OW}.\\

For some classifications of linear and quadratic $\infty$-harmonic
maps from and into a sphere, quadratic $\infty$-harmonic maps
between Euclidean spaces, linear and quadratic $\infty$-harmonic
maps between Nil and Euclidean spaces and between Sol and Euclidean
spaces see \cite{WO}.\\

In this paper, we give complete classifications of linear
$\infty$-harmonic maps between Euclidean and Heisenberg spaces,
between Nil and Sol spaces. We also classify all $\infty$-harmonic
linear automorphisms of Sol space and show that there is a subgroup
of $\infty$-harmonic linear automorphisms in the
group of linear automorphisms of Sol space.\\

\section{Linear $\infty$-harmonic maps between Euclidean and
Heisenberg spaces}

{\bf 2.1 Linear $\infty$-harmonic  maps from
Heisenberg space into a Euclidean space}.\\

 Let $\mathbb{H}_{3}$=$(\mathbb{R}^{3},g)$ denote
Heisenberg space, endowed with a left invariant metric, a
3-dimensional homogeneous metric whose group of isometries has
dimension 4. With respect to the standard coordinates $(x,y,z)$ in
$\mathbb{R}^{3}$, the metric can be written as $g={\rm d}x^{2}+{\rm
d}y^{2}+({\rm d}z+\frac{y}{2}{\rm d}x-\frac{x}{2}{\rm d}y)^{2}$
whose components are given by:
\begin{align}
& g_{11}=1+\frac{y^{2}}{4},\;
g_{12}=-\frac{xy}{4},g_{13}=\frac{y}{2},\;
g_{22}=1+\frac{x^{2}}{4},\;g_{23}=-\frac{x}{2},\;g_{33}=1;\\
&g^{11}=1,\;g^{12}=0,g^{13}=-\frac{y}{2},\;g^{22}=1,\;g^{23}=\frac{x}{2},\;g^{33}=1+\frac{x^{2}+y^{2}}{4}.\label{h3}
\end{align}
Now, let $\varphi:\mathbb{H}_{3}\longrightarrow \mathbb{R}^{n}$ with
\begin{equation}\label{MA214}
\varphi (X)=\left(\begin{array}{ccc} a_{11} & a_{12} &
a_{13}\\a_{21} & a_{22} & a_{23}\\\ldots & \ldots & \ldots\\a_{n1} &
a_{n2} & a_{n3}
\end{array}\right)\left(\begin{array}{ccc} x\\y\\z\end{array}\right)
\end{equation}
 be a linear map from Heisenberg space
into a Euclidean space. Then, we have
\begin{theorem}\label{KL254}
 A linear map $\varphi:\mathbb{H}_{3}\longrightarrow \mathbb{R}^{n}$ with $\varphi (X)=AX$, where $A$ is the
representation matrix with column vectors $A_1, A_2, A_3$, is
$\infty$-harmonic if and only if $A_3=0$, or $A_1, A_2,$ and $A_3$
are proportional to each other.
\end{theorem}
\begin{proof}
A straightforward computation using (\ref{h3}) gives:
\begin{equation}\label{00E201}
\begin{array}{lll}
\nabla\varphi^{i}=g^{\alpha\beta}\frac{\partial
\varphi^{i}}{\partial x_{\beta}}\frac{\partial }{\partial
x_{\alpha}}\\
=(a_{i1}-\frac{1}{2}a_{i3}y,\,a_{i2}+\frac{1}{2}a_{i3}x,\,\frac{1}{4}a_{i3}(x^{2}+y^{2})
+\frac{1}{2}a_{i2}x-\frac{1}{2}a_{i1}y+a_{i3}),\;\;\;\;\;\;i=1, 2,
..., n,\\\notag
\end{array}
\end{equation}
\begin{equation}\label{00E202}
\begin{array}{lll}
\left|{\rm d}\varphi\right|^{2}=g^{\alpha
\beta}{\varphi_{\alpha}}^{i}{\varphi_{\beta}}^{j}\delta_{ij}\\
=\frac{1}{4}\sum\limits_{i=1}^{n}a_{i3}^{2}x^{2}+\sum\limits_{i=1}^{n}a_{i2}a_{i3}x
+\frac{1}{4}\sum\limits_{i=1}^{n}a_{i3}^{2}y^{2}-\sum\limits_{i=1}^{n}a_{i1}a_{i3}y+\sum\limits_{j=1}^{3}\sum\limits_{i=1}^{n}a_{ij}^{2}\\\notag
\end{array}
\end{equation}
and
\begin{equation}\label{00E203}
\begin{array}{lll}
\frac{\partial \left|{\rm d}\varphi\right|^{2}}{\partial x_{1}}
=\frac{\partial \left|{\rm d}\varphi\right|^{2}}{\partial x}
=\frac{1}{2}\sum\limits_{i=1}^{n}a_{i3}^{2}x+\sum\limits_{i=1}^{n}a_{i2}a_{i3},\\
\frac{\partial \left|{\rm d}\varphi\right|^{2}}{\partial
x_{2}}=\frac{\partial \left|{\rm d}\varphi\right|^{2}}{\partial y}=
\frac{1}{2}\sum\limits_{i=1}^{n}a_{i3}^{2}y-\sum\limits_{i=1}^{n}a_{i1}a_{i3},\\
\frac{\partial \left|{\rm d}\varphi\right|^{2}}{\partial
x_{3}}=\frac{\partial \left|{\rm d}\varphi\right|^{2}}{\partial z}=0.\\

\end{array}
\end{equation}
It follows from corollary\ref{PL124} that $\varphi$ is
$\infty$-harmonic if and only if
\begin{equation}\label{00E2}
g(\nabla\,\varphi^{i}, \nabla\left|{\rm
d}\varphi\right|^{2})=0,\;\;\;\;\;\;\; i=1, 2, \ldots, n.
\end{equation}
which is equivalent to
\begin{equation}\label{N3}
\begin{array}{lll}
\frac{1}{2}(a_{i1}\sum\limits_{j=1}^{n}a_{j3}^{2}-a_{i3}\sum\limits_{j=1}^{n}a_{j1}a_{j3})x
+\frac{1}{2}(a_{i2}\sum\limits_{j=1}^{n}a_{j3}^{2}-a_{i3}\sum\limits_{j=1}^{n}a_{j2}a_{j3})y
\\+a_{i1}\sum\limits_{j=1}^{n}a_{j2}a_{j3}-a_{i2}\sum\limits_{j=1}^{n}a_{j1}a_{j3}=0
\end{array}
\end{equation}
for $i=1, 2, \ldots, n$ and for any $x, y$. By comparing the
coefficients of the polynomial identity we have
\begin{eqnarray}\label{M158}
&&
a_{i1}\sum\limits_{j=1}^{n}a_{j3}^{2}-a_{i3}\sum\limits_{j=1}^{n}a_{j1}a_{j3}=0,\;\;\;\;\;\;\;
i=1, 2, \ldots, n,
\\\label{L2}
&&
a_{i2}\sum\limits_{j=1}^{n}a_{j3}^{2}-a_{i3}\sum\limits_{j=1}^{n}a_{j2}a_{j3}=0,\;\;\;\;\;\;\;
i=1, 2, \ldots, n,
\\\label{L3}
&&
a_{i1}\sum\limits_{j=1}^{n}a_{j2}a_{j3}-a_{i2}\sum\limits_{j=1}^{n}a_{j1}a_{j3}=0
,\;\;\;\;\;\;\; i=1, 2, \ldots, n.
\end{eqnarray}

Noting that $A_i=(a_{1i},\ldots a_{ni})^t$ for $i=1, 2, 3$ are the
column vectors of $A$ we conclude that the system of equations
(\ref{M158}), (\ref{L2}), (\ref{L3}) is  equivalent to $A_1//A_3,\;
A_2//A_3$, and $A_1//A_2$, or, $A_3=0$,  from which the theorem
follows.
\end{proof}
\begin{remark}
It follows from our theorem that the maximum rank of the linear
$\infty$-harmonic map from Heisenberg space into a Euclidean space
is $2$.
\end{remark}
\begin{example}
Let $\varphi:\mathbb{H}_{3}\longrightarrow \mathbb{R}^{n},$ with
\begin{equation}\label{MA214}
\varphi (X)=\left(\begin{array}{ccc} 1 & 1 & 1\\2 & 2 & 2\\\ldots &
\ldots & \ldots\\n & n & n
\end{array}\right)\left(\begin{array}{ccc} x\\y\\z\end{array}\right).
\end{equation}
Then, by our theorem, $\varphi$ is an $\infty$-harmonic map with
non-constant energy density \\$|{\rm
d}\varphi|^2=\frac{1}{4}|A_3|^2(x^2+y^2)+|A_3|^2(x-y)+3|A_3|^2$,
where $|A_3|^2=\frac{n(n+1)(2n+1)}{6}$.
\end{example}

{\bf 2.2 Linear $\infty$-harmonic  maps from a Euclidean space
into Heisenberg space}.\\

\begin{theorem}\label{KL426}
Let $\varphi:\mathbb{R}^{m}\longrightarrow\mathbb{H}_{3}$ with
\begin{equation}\label{KL404}
\varphi (X)=\left(\begin{array}{cccc} a_{11} & a_{12} &
\ldots & a_{1m}\\
a_{21} & a_{22} & \ldots&a_{2m}\\
 a_{31} & a_{32} & \ldots &a_{3m}
\end{array}\right)\left(\begin{array}{cccc}x_{1}\\

x_{2}\\

\vdots\\

x_{m}\\

\end{array}\right)
\end{equation}
be a linear map from a Euclidean space into Heisenberg space.
 Then, $\varphi$ is $\infty$-harmonic if
and only if the row vectors $A^{1}, A^{2}$ are proportional to each
other.
\end{theorem}
\begin{proof}
A straightforward computation gives:
\begin{equation}\label{NL459}
\begin{array}{lll}
\nabla\varphi^{i} =A^{i},\;\;\;\;\;\;i=1, 2, 3,\\
\end{array}
\end{equation}
\begin{eqnarray}\label{NL202}
\left|{\rm d}\varphi\right|^{2}=
\delta^{\alpha\beta}{\varphi^{i}_{\alpha}}{\varphi^{j}_{\beta}}g_{ij}
=\frac{1}{4}|A^{2}|^{2}x^{2}+\frac{1}{4}|A^{1}|^{2}y^{2}
-\frac{1}{2}A^{1}\cdot A^{2}xy\\\notag -A^{2}\cdot A^{3}x
+A^{1}\cdot A^{3}y +(|A^{1}|^{2}+|A^{2}|^{2}+|A^{3}|^{2})\\\notag
\end{eqnarray}
\begin{eqnarray}\label{NL484}
\frac{\partial \left|{\rm d}\varphi\right|^{2}}{\partial x_{k}}
=&\frac{1}{2}(a_{1k}|A^{2}|^{2}- a_{2k}A^{1}\cdot A^{2})x+
\frac{1}{2}(a_{2k}|A^{1}|^{2}- a_{1k}A^{1}\cdot A^{2})y\\\notag
&+a_{2k}A^{1}\cdot A^{3}-a_{1k}A^{2}\cdot A^{3}, \;\;\;\;\;\;\; k=1,
2, \ldots, m.
\end{eqnarray}
It follows from corollary\ref{PL124} that $\varphi$ is
$\infty$-harmonic if and only if
\begin{equation}\label{NL2}
g(\nabla\,\varphi^{i}, \nabla\left|{\rm
d}\varphi\right|^{2})=0,\;\;\;\;\;\;\; i=1, 2, 3,
\end{equation}
which is equivalent to
\begin{equation}\label{NL503}
\begin{array}{lll}
\frac{1}{2}(A^{i}\cdot A^{1}|A^{2}|^{2}- A^{i}\cdot A^{2}A^{1}\cdot
A^{2})x+ \frac{1}{2}(A^{i}\cdot A^{2}|A^{1}|^{2}- A^{i}\cdot
A^{1}A^{1}\cdot A^{2})y\\ +A^{i}\cdot A^{2}A^{1}\cdot
A^{3}-A^{i}\cdot A^{1}A^{2}\cdot A^{3} =0 ,\;\;\;\;\;\;\; i=1, 2,
\ldots,3.
\end{array}
\end{equation}
Substituting $x=A^1X, y=A^2X$ into (\ref{NL503}) we have, for any
$X\in \mathbb{R}^m$,
\begin{equation}\label{line}
\begin{array}{lll}
(c_1A^1+ c_2A^2)X +c_3 =0
\end{array}
\end{equation}
where
\begin{eqnarray}\notag
c_1&=&\frac{1}{2}(A^{i}\cdot A^{1}|A^{2}|^{2}- A^{i}\cdot
A^{2}A^{1}\cdot A^{2}),\\\notag c_2&=&\frac{1}{2}(A^{i}\cdot
A^{2}|A^{1}|^{2}- A^{i}\cdot
A^{1}A^{1}\cdot A^{2}),\\
c_3&=&A^{i}\cdot A^{2}A^{1}\cdot A^{3}-A^{i}\cdot A^{1}A^{2}\cdot
A^{3},\;\;\;\;\;\;\; i=1, 2, 3.
\end{eqnarray}
Since Equation (\ref{line}) holds for any $X\in \mathbb{R}^m$ it can
be viewed as an identity of polynomials. It follows that $\varphi$
is $\infty$-harmonic if and only if $A^1$ and $A^2$ are proportional
to each other and $c_3=0$. One can check that $c_3=0$ is a
consequence of $A^1$ being proportional to $A^2$. Therefore, we
conclude that linear map $\varphi$ from a Euclidean space into
Heisenberg space is $\infty$-harmonic if and only if $A^1$ is
proportional to $A^2$.
\end{proof}
\begin{remark}
It follows from our theorem that the maximum rank of a linear
$\infty$-harmonic map from a Euclidean space into Heisenberg space
is $2$ and a rank $2$ linear $\infty$-harmonic map from a Euclidean
space into Heisenberg space always has non-constant energy density.
We would also like to point out that in \cite{WO} a complete
classification of linear $\infty$-harmonic maps between Euclidean
and Nil spaces is given. It is well known that Nil space is
isometric to Heisenberg space. However, as the linearity of maps
that we study depends on the (local) coordinates used in
$\mathbb{R}^3$ and since the isometry between Nil and Heisenberg
spaces is given by a quadratic polynomial map, the linear maps
between Euclidean and Nil spaces and the linear maps between
Euclidean and Heisenberg spaces are not isometric invariant and
should be treated differently as the following examples show.
\end{remark}
\begin{example}
We can check that $\sigma :(\mathbb{H}_{3}, g)\longrightarrow
(\mathbb{R}^{3},g_{Nil})$ with $\sigma (X,Y,Z)=(X, Y, Z+XY/2)$ is an
isometry from Heisenberg space onto Nil space. If we identify these
two spaces through this isometry, then the linear map
$\varphi:\mathbb{R}^{m}\longrightarrow\mathbb{H}_{3}$ with
\begin{equation}\label{KL404}
\varphi (X)=\left(\begin{array}{ccccc} 1 & -1 & 0&
\ldots & 0\\
2 & -2 & 0 & \ldots & 0\\
 0 & 0 & 0 & \ldots & 0
\end{array}\right)\left(\begin{array}{cccc}x_{1}\\

x_{2}\\

\vdots\\

x_{m}\\

\end{array}\right)
\end{equation}
becomes a quadratic map $\mathbb{R}^{m}\longrightarrow
(\mathbb{R}^{3},g_{Nil})$ with
$\sigma\circ\varphi(X)=(x_1-x_2,2(x_1-x_2),(x_1-x_2)^2)$. It is
interesting to note that the composition $\sigma\circ\varphi$ of
$\varphi$ (which is $\infty$-harmonic by Theorem \ref{KL426}) with
an isometry $\sigma$ is also $\infty$-harmonic. This follows from a
general result in  \cite{OW} that the $\infty$-harmonicity of a map
is invariant under an isometric immersion of the target space of the
map into another manifold.
\end{example}
\begin{example}
It is proved in \cite{OW} that any isometry is an $\infty$-harmonic
morphism meaning that the map preserves $\infty$-harmonicity in the
sense that it pulls back $\infty$-harmonic functions to
$\infty$-harmonic functions. One can also check that an
$\infty$-harmonic morphism pulls back $\infty$-harmonic maps to
$\infty$-harmonic maps. It follows that the isometry $\sigma
:(\mathbb{H}_{3}, g)\longrightarrow (\mathbb{R}^{3},g_{Nil})$ with
$\sigma (X,Y,Z)=(X, Y, Z+XY/2)$ is an $\infty$-harmonic morphism. By
\cite{WO}, the linear map
$\varphi:(\mathbb{R}^{3},g_{Nil})\longrightarrow \mathbb{R}^{n}\;
(n\ge 2)$
\begin{equation}\label{MAP1}
\varphi (X)=\left(\begin{array}{ccc} 0 & a_{12} & a_{13}\\0 & a_{22}
& a_{23}\\\ldots& \ldots & \ldots\\0 & a_{n2} & a_{n3}
\end{array}\right)\left(\begin{array}{ccc} x\\y\\z\end{array}\right).
\end{equation}
is $\infty$-harmonic. Therefore, the composition
$\varphi\circ\sigma:(\mathbb{H}_{3}, g)\longrightarrow
\mathbb{R}^{n}$ given by
\begin{equation}\label{MAP2}
\varphi \circ \sigma(X)=\left(\begin{array}{ccc} 0 & a_{12} &
a_{13}\\0 & a_{22} & a_{23}\\\ldots& \ldots & \ldots\\0 & a_{n2} &
a_{n3}
\end{array}\right)\left(\begin{array}{ccc} x\\y\\z+\frac{1}{2}xy\end{array}\right)
\end{equation}
gives an $\infty$-harmonic map defined by polynomials of degree $2$
from Heisenberg space into a Euclidean space with constant energy
density.
\end{example}

\section{Linear $\infty$-harmonic  maps between Nil and Sol spaces}

In this section we give a complete classification of linear
$\infty$-harmonic maps between Nil and Sol spaces. It turns out that
the maximum rank of linear $\infty$-harmonic maps between Nil and
Sol spaces is $2$ and some of them have constant energy density
while others may have non-constant energy density.\\

 {\bf 3.1 Linear $\infty$-harmonic maps from Nil space into Sol space}.\\

Let $(\mathbb{R}^{3}, g_{Nil})$ and $(\mathbb{R}^{3},g_{Sol})$
denote Nil and Sol spaces, where the metrics with respect to the
standard coordinates $(x,y,z)$ in $\mathbb{R}^{3}$ are given by
$g_{Nil}={\rm d}x^{2}+{\rm d}y^{2}+({\rm d}z-x{\rm d}y)^{2}$ and
$g_{Sol}={e^{2z}}{\rm d}x^{2}+{e^{-2z}}{\rm d}y^{2}+{\rm d}{z}^{2}$
respectively. In the following, we use the notations $g=g_{Nil}$,
$h=g_{Sol}$, the coordinates $\{x, y, z\}$  in
$(\mathbb{R}^{3},g_{Nil})$ and the coordinates
 $\{x\acute{}, y\acute{}, z\acute{}\;\}$  in $(\mathbb{R}^{3},g_{Sol})$, then one can easily compute the following
components of Nil and Sol metrics:
\begin{align}\nonumber
& g_{11}=1,\; g_{12}=g_{13}=0,\;
g_{22}=1+x^{2},\;g_{23}=-x,\;g_{33}=1;\\\nonumber
&g^{11}=1,\;g^{12}=g^{13}=0,\;g^{22}=1,\;g^{23}=x,\;g^{33}=1+x^{2}.\\\notag
&h_{11}=e^{2z\acute{}},\;h_{22}=e^{-2z\acute{}},\;h_{33}=1,\;{\rm all\; other},\;h_{ij}=0;\\\
&h^{11}=e^{-2z\acute{}},\;h^{22}=e^{2z\acute{}},\;h^{33}=1,\;{\rm
all\; other},\;h^{ij}=0.\notag
\end{align}
Now we study the $\infty$-harmonicity of linear maps between Nil and
Sol spaces. First, we give the following classification of linear
$\infty$-harmonic maps from Nil space into Sol space.

\begin{theorem}\label{KL1043}
A linear map $\varphi:(\mathbb{R}^{3},g_{Nil})\longrightarrow
(\mathbb{R}^{3}, g_{Sol})$ from Nil space into Sol space with
\begin{equation}\label{MN942}
\varphi (X)=\left(\begin{array}{ccc} a_{11} & a_{12} &
a_{13}\\a_{21} & a_{22} & a_{23}\\ a_{31} & a_{32} & a_{33}\\
\end{array}\right)\left(\begin{array}{ccc}
x\\y\\z\end{array}\right)
\end{equation}
is $\infty$-harmonic if and only if $\varphi$ takes one of the
following forms:
\begin{equation}\label{MAP82}
\varphi (X)=\left(\begin{array}{ccc}
0 & a_{12} & a_{13}\\
0 & a_{22} & a_{23}\\
0& 0 & 0
\end{array}\right)\left(\begin{array}{ccc}
x\\y\\z\end{array}\right),
\end{equation}
\begin{equation}\label{MAP83}
\varphi (X)=\left(\begin{array}{ccc}
a_{11} & a_{12} & 0\\
a_{21} & a_{22} & 0\\
0& 0 & 0
\end{array}\right)\left(\begin{array}{ccc}
x\\y\\z\end{array}\right),
\end{equation}
\begin{equation}\label{MAP80}
\varphi (X)=\left(\begin{array}{ccc}
0 & 0 & 0\\
0 & 0 & 0\\
0&a_{32}&a_{33}
\end{array}\right)\left(\begin{array}{ccc}
x\\y\\z\end{array}\right),\;\;\; {\rm or}
\end{equation}

\begin{equation}\label{MAP81}
\varphi (X)=\left(\begin{array}{ccc}
0 & 0 & 0\\
0 & 0 & 0\\
a_{31}&a_{32}&0
\end{array}\right)\left(\begin{array}{ccc} x\\y\\z\end{array}\right).
\end{equation}
\end{theorem}
\begin{proof}
A straightforward computation gives:
\begin{equation}\label{00E2010}
\begin{array}{lll}
\nabla\varphi^{i}=g^{\alpha \beta}\frac{\partial
\varphi^{i}}{\partial x_{\beta}}
\frac{\partial }{\partial x_{\alpha}} \\
=(a_{i1},\,a_{i3}x+a_{i2},\,a_{i3}x^{2}+a_{i2}x+a_{i3}),\;\;\;\;\;\;i=1,
2, 3.\\\notag
\end{array}
\end{equation}
and
\begin{eqnarray}\notag
\left|{\rm d}\varphi\right|^{2}=g^{\alpha
\beta}{\varphi_{\alpha}}^{i}{\varphi_{\beta}}^{j}h_{ij}\circ\varphi\\\notag
=(a_{13}^{2}x^{2}+2a_{12}a_{13}x+\sum\limits_{j=1}^{3}a_{1j}^{2})e^{2z\acute{}}\\\label{ENER}
+(a_{23}^{2}x^{2}+2a_{22}a_{23}x+\sum\limits_{j=1}^{3}a_{2j}^{2})e^{-2z\acute{}}\\\notag
+(a_{33}^{2}x^{2}+2a_{32}a_{33}x+\sum\limits_{j=1}^{3}a_{3j}^{2}),\;\;\\\notag
\
\end{eqnarray}
where $z\acute{}=a_{31}x+a_{32}y+a_{33}z$.\\

Also, one can check that
\begin{equation}\label{00E2030}
\begin{array}{lll}
\frac{\partial \left|{\rm d}\varphi\right|^{2}}{\partial x_{1}}
=\frac{\partial \left|{\rm d}\varphi\right|^{2}}{\partial x}\\
=2\{a_{31}a_{13}^{2}x^{2}+(a_{13}^{2}+2a_{31}a_{12}a_{13})x
+a_{12}a_{13}+a_{31}\sum\limits_{j=1}^{3}a_{1j}^{2}\}e^{2z\acute{}}\\
-2\{a_{31}a_{23}^{2}x^{2}+(2a_{31}a_{22}a_{23}-a_{23}^{2})x
+a_{31}\sum\limits_{j=1}^{3}a_{2j}^{2}-a_{22}a_{23}\}e^{-2z\acute{}}\\
+2(a_{33}^{2}x+a_{32}a_{33}),\\
\frac{\partial \left|{\rm d}\varphi\right|^{2}}{\partial x_{2}}
=\frac{\partial \left|{\rm d}\varphi\right|^{2}}{\partial y}\\
=2a_{32}(a_{13}^{2}x^{2}+2a_{12}a_{13}x+\sum\limits_{j=1}^{3}a_{1j}^{2})e^{2z\acute{}}\\
-2a_{32}(a_{23}^{2}x^{2}+2a_{22}a_{23}x+\sum\limits_{j=1}^{3}a_{2j}^{2})e^{-2z\acute{}},\\
\frac{\partial \left|{\rm d}\varphi\right|^{2}}{\partial x_{3}}
=\frac{\partial \left|{\rm d}\varphi\right|^{2}}{\partial z}\\
=2a_{33}(a_{13}^{2}x^{2}+2a_{12}a_{13}x+\sum\limits_{j=1}^{3}a_{1j}^{2})e^{2z\acute{}}\\
-2a_{33}(a_{23}^{2}x^{2}+2a_{22}a_{23}x+\sum\limits_{j=1}^{3}a_{2j}^{2})e^{-2z\acute{}}.
\end{array}
\end{equation}
It follows from corollary \ref{PL124} that $\varphi$ is an
$\infty$-harmonic map if and only if
\begin{equation}\label{MAP6}
g(\nabla\,\varphi^{i}, \nabla\left|{\rm
d}\varphi\right|^{2})=0,\;\;\;\;\;\;\; i=1, 2, 3.
\end{equation}
which is equivalent to
\begin{equation}\label{N3669}
\begin{array}{lll}
2\{a_{i3}a_{33}a_{13}^{2}x^{4}+[(a_{i3}a_{32}+a_{i2}a_{33})a_{13}^{2}
+2a_{i3}a_{33}a_{12}a_{13}]x^{3}\\
+[\sum\limits_{k=1}^{3}a_{ik}a_{3k}a_{13}^{2}
+2(a_{i3}a_{32}+a_{i2}a_{33})a_{12}a_{13}
+a_{i3}a_{33}\sum\limits_{j=1}^{3}a_{1j}^{2}]x^{2}\\
+[(a_{i3}a_{32}+a_{i2}a_{33})\sum\limits_{j=1}^{3}a_{1j}^{2}
+2\sum\limits_{k=1}^{3}a_{ik}a_{3k}a_{12}a_{13}+a_{i1}a_{13}^{2}]x\\
+[\sum\limits_{k=1}^{3}a_{ik}a_{3k}\sum\limits_{j=1}^{3}a_{1j}^{2}
+a_{i1}a_{12}a_{13}]\}e^{2z\acute{}}\\
-2\{a_{i3}a_{33}a_{23}^{2}x^{4}+[(a_{i3}a_{32}+a_{i2}a_{33})a_{23}^{2}
+2a_{i3}a_{33}a_{22}a_{23}]x^{3}\\
+[\sum\limits_{k=1}^{3}a_{ik}a_{3k}a_{23}^{2}+2(a_{i3}a_{32}+a_{i2}a_{33})a_{22}a_{23}
+a_{i3}a_{33}\sum\limits_{j=1}^{3}a_{2j}^{2}]x^{2}\\
+[(a_{i3}a_{32}+a_{i2}a_{33})\sum\limits_{j=1}^{3}a_{2j}^{2}
+2\sum\limits_{k=1}^{3}a_{ik}a_{3k}a_{22}a_{23}-a_{i1}a_{23}^{2}]x\\
+[\sum\limits_{k=1}^{3}a_{ik}a_{3k}\sum\limits_{j=1}^{3}a_{2j}^{2}
-a_{i1}a_{22}a_{23}]\}e^{-2z\acute{}}
+2a_{i1}(a_{33}^{2}x+a_{32}a_{33})=0,\\i=1,2,3.
\end{array}
\end{equation}
Case (A): $\sum\limits_{j=1}^{3}a_{3j}^{2}=0$. In this case, (\ref
{N3669}) becomes
\begin{equation}\label{N3693}
\begin{array}{lll}
2a_{i1}(a_{13}^{2}+ a_{23}^{2})x
+2a_{i1}(a_{12}a_{13}+a_{22}a_{23})=0 ,\;\;\;\;\;\;i=1,2,3.
\end{array}
\end{equation}
Solving Equation (\ref{N3693}), we have $a_{i1}=0$ for $i=1,2,3$, or
$a_{13}=a_{23}=0 $. These give the classes of linear $\infty$-harmonic maps corresponding to (\ref{MAP82}) and (\ref{MAP83}).\\

Case (B): $\sum\limits_{j=1}^{3}a_{3j}^{2}\neq 0$. In this case, we
use the fact that the functions \\$1,\; x,\; xe^{2x},\; x^2e^{2x},\;
x^3e^{2x}, \;x^4e^{2x};\; xe^{-2x},\; x^2e^{-2x},\; x^3e^{-2x},
\;x^4e^{-2x}$ are linearly independence to conclude that (\ref
{N3669}) is equivalent to
\begin{equation}\label{NSOL}
\left\{\begin{array}{rl}
  a_{i1}a_{33}^{2}=0, \;\;\langle 1\rangle\\
 a_{i1}a_{32}a_{33}=0, \;\;\langle 2\rangle\\
 a_{i3}a_{33}a_{13}^{2}=0, \;\;\langle 3\rangle\\
  (a_{i3}a_{32}+a_{i2}a_{33})a_{13}^{2}+2a_{i3}a_{33}a_{12}a_{13}=0, \;\;\langle 4\rangle\\
 \sum\limits_{k=1}^{3}a_{ik}a_{3k}a_{13}^{2}
+2(a_{i3}a_{32}+a_{i2}a_{33})a_{12}a_{13}
+a_{i3}a_{33}\sum\limits_{j=1}^{3}a_{1j}^{2}=0, \;\;\langle 5\rangle\\
 (a_{i3}a_{32}+a_{i2}a_{33})\sum\limits_{j=1}^{3}a_{1j}^{2}
+2\sum\limits_{k=1}^{3}a_{ik}a_{3k}a_{12}a_{13}+a_{i1}a_{13}^{2}=0, \;\;\langle 6\rangle\\
 \sum\limits_{k=1}^{3}a_{ik}a_{3k}\sum\limits_{j=1}^{3}a_{1j}^{2}
+a_{i1}a_{12}a_{13} =0, \;\;\langle 7\rangle\\
  a_{i3}a_{33}a_{23}^{2}=0, \;\;\langle 8\rangle\\
 (a_{i3}a_{32}+a_{i2}a_{33})a_{23}^{2}
+2a_{i3}a_{33}a_{22}a_{23}=0, \;\;\langle 9\rangle\\
\sum\limits_{k=1}^{3}a_{ik}a_{3k}a_{23}^{2}+2(a_{i3}a_{32}+a_{i2}a_{33})a_{22}a_{23}
+a_{i3}a_{33}\sum\limits_{j=1}^{3}a_{2j}^{2}=0, \;\;\langle 10\rangle\\
  (a_{i3}a_{32}+a_{i2}a_{33})\sum\limits_{j=1}^{3}a_{2j}^{2}
+2\sum\limits_{k=1}^{3}a_{ik}a_{3k}a_{22}a_{23}-a_{i1}a_{23}^{2}=0. \;\;\langle 11\rangle\\
 \sum\limits_{k=1}^{3}a_{ik}a_{3k}\sum\limits_{j=1}^{3}a_{2j}^{2}
-a_{i1}a_{22}a_{23}=0. \;\;\langle12\rangle
\end{array}\right.
\end{equation}
It follows from $\langle 1\rangle$ of (\ref{NSOL}) that either $
a_{i1}=0$ for $i=1, 2, 3$, or $a_{33}=0$.\\
Case ($B_1$): $\sum\limits_{i=1}^{3}a_{i1}^{2}=0$. In this case, we
have $a_{32}^{2}+a_{33}^{2}\neq 0$ since we are in Case (B). It
follows that the Equations $\langle7\rangle$ and $\langle12\rangle$ of (\ref {NSOL}) reduce to be\\
\begin{equation}\label{N3658}
\left\{\begin{array}{rl}
(a_{i3}a_{33}+a_{i2}a_{32})\sum\limits_{j=1}^{3}a_{1j}^{2}=0\\
(a_{i3}a_{33}+a_{i2}a_{32})\sum\limits_{j=1}^{3}a_{2j}^{2}=0,\;\;i=1, 2, 3.\\
\end{array}\right.
\end{equation}
Writing out the Equation (\ref{N3658}) with $i=3$ we have that
$a_{1j}=a_{2j}=0$ for $j=1,2,3$ and we can check that these,
together with $a_{j1}=0$, are solutions of the Equations (\ref
{NSOL}). These correspond to the class of linear $\infty$-harmonic maps given by (\ref{MAP80}).\\
 Case ($B_2$): $a_{33}=0$ and hence $a_{31}^{2}+a_{32}^{2}\neq 0$ since we are in Case (B).\\
In this case, Equation (\ref{NSOL}) reduces to
\begin{equation}\label{N3670}
\left\{\begin{array}{rl}
a_{i3}a_{32}a_{13}^{2}=0,\\
(a_{i2}a_{32}+a_{i1}a_{31})a_{13}^{2} +2a_{i3}a_{32}a_{12}a_{13}=0, \\
a_{i3}a_{32}\sum\limits_{j=1}^{3}a_{1j}^{2}
+2(a_{i2}a_{32}+a_{i1}a_{31})a_{12}a_{13}+a_{i1}a_{13}^{2}=0, \\
(a_{i2}a_{32}+a_{i1}a_{31})\sum\limits_{j=1}^{3}a_{1j}^{2}
+a_{i1}a_{12}a_{13} =0,\\
a_{i3}a_{32}a_{23}^{2}=0,\\
(a_{i2}a_{32}+a_{i1}a_{31})a_{23}^{2} +2a_{i3}a_{32}a_{22}a_{23}=0,\\
a_{i3}a_{32}\sum\limits_{j=1}^{3}a_{2j}^{2}
+2(a_{i2}a_{32}+a_{i1}a_{31})a_{22}a_{23}-a_{i1}a_{23}^{2}=0, \\
(a_{i2}a_{32}+a_{i1}a_{31})\sum\limits_{j=1}^{3}a_{2j}^{2}
-a_{i1}a_{22}a_{23} =0,
\end{array}\right.\;\;\;i=1,2,3.
\end{equation}
It follows from the first equation of (\ref {N3670})that we either
have $a_{13}=0$ or $a_{32}=0$. By considering
following cases:\\
(I) $a_{13}=0,a_{32}\neq 0$,\\
(II) $a_{13}\neq 0,a_{32}= 0$, hence, $a_{31}\neq 0$ \\
(III) $a_{13}= 0,a_{32}= 0$, hence, $a_{31}\neq 0$ \\
we obtain that $a_{i3}=0,a_{1i}=a_{2i}=0,$ \;for \; $i=1, 2, 3,$ are
solution of the Equations (\ref {NSOL}), which give the class of
linear $\infty$-harmonic maps corresponding to (\ref
{MAP81}).\\
 Thus, we obtain the theorem.
\end{proof}
\begin{remark}
It follows from our theorem that the maximum rank of linear
$\infty$-harmonic maps from Nil into Sol is $2$. Using the energy
density formula (\ref{ENER}) we can check that some of them have
non-constant energy density while others have constant energy
density.
\end{remark}

 {\bf 3.2 Linear $\infty$-harmonic  maps from Sol space
 into Nil space}.\\

 The linear $\infty$-harmonic maps from Sol space into Nil space can be completely described by the following theorem.
\begin{theorem}\label{KL1043}
A linear map $\varphi:(\mathbb{R}^{3}, g_{Sol})\longrightarrow
(\mathbb{R}^{3},g_{Nil})$ from Sol space into Nil space with
\begin{equation}\label{MN942}
\varphi (X)=\left(\begin{array}{ccc} a_{11} & a_{12} &
a_{13}\\a_{21} & a_{22} & a_{23}\\ a_{31} & a_{32} & a_{33}\\
\end{array}\right)\left(\begin{array}{ccc}
x\\y\\z\end{array}\right)
\end{equation}
is $\infty$-harmonic if and only if $\varphi$ takes one of the
following forms:
\begin{equation}\label{MAP192185}
\varphi (X)=\left(\begin{array}{ccc}
0 & 0 & 0\\
a_{21} & a_{22} & 0\\
a_{31}&a_{32}&0
\end{array}\right)\left(\begin{array}{ccc}
x_{1}\\x_{2}\\x_{3}\end{array}\right),
\end{equation}

\begin{equation}\label{MAP192200}
\varphi (X)=\left(\begin{array}{ccc}
a_{11}& a_{12} & 0\\
0 & 0 & 0\\
a_{31}&a_{32}&0
\end{array}\right)\left(\begin{array}{ccc}
x_{1}\\x_{2}\\x_{3}\end{array}\right),
\end{equation}
\begin{equation}\label{MAP192193}
\varphi (X)=\left(\begin{array}{ccc}
0 &0 & 0\\
0 & 0 &a_{23}\\
0 &0&a_{33}
\end{array}\right)\left(\begin{array}{ccc}
x_{1}\\x_{2}\\x_{3}\end{array}\right),\; {\rm or}
\end{equation}

\begin{equation}\label{MAP192208}
\varphi (X)=\left(\begin{array}{ccc}
0 &0 & a_{13}\\
0 & 0 &0\\
0 &0&a_{33}
\end{array}\right)\left(\begin{array}{ccc} x_{1}\\x_{2}\\x_{3}\end{array}\right).
\end{equation}
\end{theorem}

\begin{proof}
 Using the notations $g=g_{Sol}$, $h=g_{Nil}$, and the coordinates $\{x_1, x_2, x_3\}$  in $(\mathbb{R}^{3},g_{Sol})$ and
 $\{y_1, y_2, y_3\}$  in $(\mathbb{R}^{3},g_{Nil})$ we compute the
following components of Nil and Sol metric:
\begin{align}\nonumber
&g_{11}=e^{2x_{3}},\;g_{22}=e^{-2x_{3}},\;g_{33}=1,\;{\rm all\; other},\;g_{ij}=0,;\\\
&g^{11}=e^{-2x_{3}},\;g^{22}=e^{2x_{3}},\;g^{33}=1,\;{\rm all\;
other},\;g^{ij}=0.\\\notag & h_{11}=1,\; h_{12}=g_{13}=0,\;
h_{22}=1+y_{1}^{2},\;h_{23}=-y_{1},\;h_{33}=1;\\\nonumber
&h^{11}=1,\;h^{12}=h^{13}=0,\;h^{22}=1,\;h^{23}=y_{1},\;h^{33}=1+y_{1}^{2}.\notag
\end{align}

A straightforward computation gives:
\begin{equation}
\begin{array}{lll}
\nabla\varphi^{i}=g^{\alpha \beta}\frac{\partial \varphi^{i}}{\partial x_{\beta}}
\frac{\partial }{\partial x_{\alpha}} \\
=(a_{i1}e^{-2x_{3}},\,a_{i2}e^{2x_{3}},\,a_{i3}),\;\;\;\;\;\;i=1, 2,
3,
\end{array}
\end{equation}
and
\begin{equation}\label{00E1991}
\begin{array}{lll}
\left|{\rm d}\varphi\right|^{2}=g^{\alpha
\beta}{\varphi_{\alpha}}^{i}{\varphi_{\beta}}^{j}h_{ij}\circ \varphi\\
=(e^{-2x_{3}}a_{21}^{2}+e^{2x_{3}}a_{22}^{2}+a_{23}^{2})y_{1}^{2}
-2(e^{-2x_{3}}a_{21}a_{31}+e^{2x_{3}}a_{22}a_{32}+a_{23}a_{33})y_{1}\\
+\sum\limits_{i=1}^{3}a_{i1}^{2}e^{-2x_{3}}+\sum\limits_{i=1}^{3}a_{i2}^{2}e^{2x_{3}}
+\sum\limits_{i=1}^{3}a_{i3}^{2},
\end{array}
\end{equation}
where $y_{1}=a_{11}x_{1}+a_{12}x_{2}+a_{13}x_{3}$.\\
 Also, we can check that
\begin{equation*}
\begin{array}{lll}
\frac{\partial \left|{\rm d}\varphi\right|^{2}}{\partial x_{1}}
=\\2(a_{11}a_{21}^{2}e^{-2x_{3}}+a_{11}a_{22}^{2}e^{2x_{3}}+a_{11}a_{23}^{2})y_{1}\\
-2(a_{11}a_{21}a_{31}e^{-2x_{3}}+a_{11}a_{22}a_{32}e^{2x_{3}}+a_{11}a_{23}a_{33}),
\end{array}
\end{equation*}
\begin{equation*}
\begin{array}{lll}
\frac{\partial \left|{\rm d}\varphi\right|^{2}}{\partial x_{2}}
=\\2(a_{12}a_{21}^{2}e^{-2x_{3}}+a_{12}a_{22}^{2}e^{2x_{3}}+a_{12}a_{23}^{2})y_{1}\\
-2(a_{12}a_{21}a_{31}e^{-2x_{3}}+a_{12}a_{22}a_{32}e^{2x_{3}}+a_{12}a_{23}a_{33}),
\end{array}
\end{equation*}
and
\begin{equation*}
\begin{array}{lll}
\frac{\partial \left|{\rm d}\varphi\right|^{2}}{\partial x_{3}}
=\\2(a_{22}^{2}e^{2x_{3}}-a_{21}^{2}e^{-2x_{3}})y_{1}^{2}\\
+2\{(a_{13}a_{21}^{2}+2a_{21}a_{31})e^{-2x_{3}}
+(a_{13}a_{22}^{2}-2a_{22}a_{32})e^{2x_{3}}
+a_{13}a_{23}^{2}\}y_{1}\\
-2\{(\sum\limits_{i=1}^{3}a_{i1}^{2}+a_{13}a_{21}a_{31})e^{-2x_{3}}+
(a_{13}a_{22}a_{32}-\sum\limits_{i=1}^{3}a_{i2}^{2}) e^{2x_{3}}
+a_{13}a_{23}a_{33})\}.
\end{array}
\end{equation*}
Using Corollary \ref{PL124} we conclude that $\varphi$ is an
$\infty$-harmonic if and only if
\begin{eqnarray}\notag
&&2(a_{i3}a_{22}^{2}e^{2x_{3}}-a_{i3}a_{21}^{2}e^{-2x_{3}})y_{1}^{2}\\\notag
&&+2\{a_{i1}a_{11}a_{21}^{2}e^{-4x_{3}}+a_{i2}a_{12}a_{22}^{2}e^{4x_{3}}\\\notag
&&+(a_{i1}a_{11}a_{23}^{2}+a_{i3}a_{13}a_{21}^{2}+2a_{i3}a_{21}a_{31}^{2})e^{-2x_{3}}\\\notag
&&+(a_{i2}a_{12}a_{23}^{2}+a_{i3}a_{13}a_{22}^{2}-2a_{i3}a_{22}a_{32}^{2})e^{2x_{3}}\\\label{N1991}
&&+(a_{i1}a_{11}a_{22}^{2}+a_{i2}a_{12}a_{21}^{2}+a_{i3}a_{13}a_{23}^{2}\}y_{1}\\\notag
&&-2\{a_{i1}a_{11}a_{21}a_{31}e^{-4x_{3}}+a_{i2}a_{12}a_{22}a_{32}e^{4x_{3}}\\\notag
&&+(a_{i1}a_{11}a_{23}a_{33}+a_{i3}a_{13}a_{21}a_{31}
+a_{i3}\sum\limits_{j=1}^{3}a_{j1}^{2})e^{-2x_{3}}\\\notag
&&+(a_{i2}a_{12}a_{23}a_{33}+a_{i3}a_{13}a_{22}a_{32}
-a_{i3}\sum\limits_{j=1}^{3}a_{j2}^{2})e^{2x_{3}}\\\notag
&&+a_{i1}a_{11}a_{22}a_{32}+a_{i2}a_{12}a_{21}a_{31}+a_{i3}a_{13}a_{23}a_{33}\}=0,\;\;i=1,
2,3.\\\notag
\end{eqnarray}
Case (A): $\sum\limits_{j=1}^{3}a_{1j}^{2}=0$. It follows that
$y_{1}=a_{11}x_{1}+a_{12}x_{2}+a_{13}x_{3}=0$ and Equation
(\ref{N1991}) reduces to
\begin{equation}\label{00E11940}
\left\{\begin{array}{rl}
a_{i3}\sum\limits_{j=1}^{3}a_{j1}^{2}e^{-2x_{3}}=0,\\
a_{i3}\sum\limits_{j=1}^{3}a_{j2}^{2}e^{2x_{3}}=0,
\end{array}\right.\;\;\;\;\;i=1,2,3,
\end{equation}
which has solutions
$a_{i3}=0,\;or,\;a_{i1}=a_{i2}=0,\;\;for,\;i=1,2,3$. These give the
linear $\infty$-harmonic maps defined by (\ref{MAP192185}) and
(\ref{MAP192193}).\\

Case (B) $\sum\limits_{j=1}^{3}a_{1j}^{2}\neq 0$. In this case, we
use Equation (\ref{N1991}) and the fact that the functions $1,
te^{2t},t^2e^{2t}; te^{-2t}, t^2e^{-2t}; te^{4t},t^2e^{4t};
te^{-4t}, t^2e^{-4t}$ are linearly independence to conclude that
$\varphi$ is $\infty$-harmonic if and only if
\begin{equation}\label{MN1385}
\left\{\begin{array}{rl}
a_{i3}a_{22}^{2}=0\\
 a_{i3}a_{21}^{2}=0,
\end{array}\right.\;\;\;\;\;\;\;
i=1, 2,3,
\end{equation}
\begin{equation}\label{MN1393}
\left\{\begin{array}{rl}
a_{i1}a_{11}a_{21}^{2}=0\\
a_{i2}a_{12}a_{22}^{2}=0\\
a_{i1}a_{11}a_{23}^{2}+a_{i3}a_{13}a_{21}^{2}
+2a_{i3}a_{21}a_{31}^{2}=0\\
a_{i2}a_{12}a_{23}^{2}+a_{i3}a_{13}a_{22}^{2}
-2a_{i3}a_{22}a_{32}^{2}=0\\
a_{i1}a_{11}a_{22}^{2}-a_{i2}a_{12}a_{21}^{2}
+a_{i3}a_{13}a_{23}^{2}=0,\\
\end{array}\right.\;\;\;\;\;\;\;
i=1, 2,3,
\end{equation}
and
\begin{equation}\label{MN1407}
\left\{\begin{array}{rl}
a_{i1}a_{11}a_{21}a_{31}=0\\
a_{i2}a_{12}a_{22}a_{32}=0\\
a_{i1}a_{11}a_{23}a_{33}+a_{i3}a_{13}a_{21}a_{31}
+a_{i3}\sum\limits_{j=1}^{3}a_{j1}^{2}=0\\
a_{i2}a_{12}a_{23}a_{33}+a_{i3}a_{13}a_{22}a_{32}
-a_{i3}\sum\limits_{j=1}^{3}a_{j2}^{2}=0\\
a_{i1}a_{11}a_{22}a_{32}+a_{i2}a_{12}a_{21}a_{31}+a_{i3}a_{13}a_{23}a_{33}=0
,\\
\end{array}\right.\;\;\;\;\;\;\; i=1, 2,3.
\end{equation}
In this case, it is  easy to check that $a_{2i}=a_{i3}=0$ or
$a_{i1}=a_{i2}=a_{2i}=0$, $\;for\;i=1, 2, 3,$ are solutions of
system (\ref{MN1385}), (\ref{MN1393}) and (\ref{MN1407}). These give
the linear $\infty$-harmonic maps defined by (\ref{MAP192200}) and
(\ref{MAP192208}). Thus, we obtain the theorem.
\end{proof}
\begin{remark}
Again, we remark that the maximum rank of linear $\infty$-harmonic
maps from Sol into Nil is $2$. Using the energy density formula
(\ref{00E1991}) we can check that all rank $2$ linear
$\infty$-harmonic maps from Sol into Nil have non-constant energy
density.
\end{remark}

\section{$\infty$-Harmonic linear endomorphisms of Sol space}

In this final section, we study the $\infty$-harmonicity of linear
endomorphisms of Sol space. We give a complete classification of
$\infty$-harmonic linear endomorphisms of Sol space. It turns out
that an $\infty$-harmonic linear endomorphism of Sol space can have
maximum rank, i.e., there are $\infty$-harmonic linear
diffeomorphisms from Sol space onto itself which have constant
energy density and which are not isometries. We also show that there
is a subgroup of $\infty$-harmonic linear automorphisms in the group
of linear isomorphisms.
\begin{theorem}\label{KL1534}
 A linear endomorphism $\varphi:(\mathbb{R}^{3},g_{Sol})\longrightarrow
(\mathbb{R}^{3}, g_{Sol})$ of Sol space with
\begin{equation}\label{MN942}
\varphi (X)=\left(\begin{array}{ccc} a_{11} & a_{12} &
a_{13}\\a_{21} & a_{22} & a_{23}\\ a_{31} & a_{32} & a_{33}\\
\end{array}\right)\left(\begin{array}{ccc}
x\\y\\z\end{array}\right)
\end{equation}
is $\infty$-harmonic if and only if $\varphi$ takes one of the
following forms:
\begin{equation}\label{MAP8500}
\varphi (X)=\left(\begin{array}{ccc}
a_{11} & 0 & 0\\
0 & a_{22} & 0\\
0& 0 & 1
\end{array}\right)\left(\begin{array}{ccc}
x\\y\\z\end{array}\right),
\end{equation}
\begin{equation}\label{MAP8600}
\varphi (X)=\left(\begin{array}{ccc}
0 & a_{12} & 0\\
a_{21} & 0 & 0\\
0& 0 & -1
\end{array}\right)\left(\begin{array}{ccc}
x\\y\\z\end{array}\right),
\end{equation}
\begin{equation}\label{MAP8300}
\varphi (X)=\left(\begin{array}{ccc}
a_{11} & a_{12} & 0\\
a_{21} & a_{22} & 0\\
0& 0 & 0
\end{array}\right)\left(\begin{array}{ccc}
x\\y\\z\end{array}\right),
\end{equation}
\begin{equation}\label{MAP8400}
\varphi (X)=\left(\begin{array}{ccc}
0 & 0& 0\\
0 & 0& 0\\
0 & 0 & a_{33}
\end{array}\right)\left(\begin{array}{ccc}
x\\y\\z\end{array}\right),
\end{equation}
\begin{equation}\label{MAP8100}
\varphi (X)=\left(\begin{array}{ccc}
0 & 0 & 0\\
0 & 0 & 0\\
a_{31}&a_{32}&0
\end{array}\right)\left(\begin{array}{ccc}
x\\y\\z\end{array}\right), \; {\rm or}
\end{equation}
\begin{equation}\label{MAP8200}
\varphi (X)=\left(\begin{array}{ccc}
0 & 0& a_{13}\\
0 & 0& a_{23}\\
0& 0 & 0
\end{array}\right)\left(\begin{array}{ccc} x\\y\\z\end{array}\right).
\end{equation}
\end{theorem}
\begin{proof}
We use $g$ and $h$ to denote the metrics in the domain and the
target manifolds respectively. With respect to the coordinates $\{x,
y, z\}$ in the domain and $\{x\acute{}, y\acute{}, z\acute{}\;\}$ in
the target manifold, we one can easily write down the following
components of metrics:
\begin{align*}\nonumber
&g_{11}=e^{2z},\;g_{22}=e^{-2z},\;g_{33}=1,\;{\rm all\; other},\;g_{ij}=0;\\\
&g^{11}=e^{-2z},\;g^{22}=e^{2z},\;g^{33}=1,\;{\rm all\;
other},\;g^{ij}=0.\\
&h_{11}=e^{2z\acute{}},\;h_{22}=e^{-2z\acute{}},\;h_{33}=1,\;{\rm all\; other},\;h_{ij}=0;\\\
&h^{11}=e^{-2z\acute{}},\;h^{22}=e^{2z\acute{}},\;h^{33}=1,\;{\rm
all\; other},\;h^{ij}=0.\notag
\end{align*}
A direct computation gives:
\begin{equation}
\begin{array}{lll}
\nabla\varphi^{i} =g^{\alpha \beta}\frac{\partial
\varphi^{i}}{\partial x_{\beta}}\frac{\partial}{\partial
x_{\alpha}}\\
=(a_{i1}e^{-2z},\,a_{i2}e^{2z},\,a_{i3}),\;\;\;\;\;\;i=1, 2,
3,\\\notag
\end{array}
\end{equation}
and
\begin{eqnarray}\label{SOLS}
&&\left|{\rm d}\varphi\right|^{2}= g^{\alpha
\beta}{\varphi_{\alpha}}^{i}{\varphi_{\beta}}^{j}
h_{ij}\circ\varphi=g^{\alpha \alpha}(\frac{\partial
\varphi^{i}}{\partial x_{\alpha}})^{2}h_{ii}\circ\varphi\\\notag
&&=(a_{12}^{2}e^{2z}+a_{11}^{2}e^{-2z}+a_{13}^{2})e^{2z\acute{}}\\\notag
&&+(a_{22}^{2}e^{2z}+a_{21}^{2}e^{-2z}+a_{23}^{2})e^{-2z\acute{}}
+(a_{32}^{2}e^{2z}+a_{31}^{2}e^{-2z}+a_{33}^{2}),
\end{eqnarray}
where $z\acute{}=a_{31}x+a_{32}y+a_{33}z$. Furthermore, we compute
that
\begin{equation}
\begin{array}{lll}
\frac{\partial \left|{\rm d}\varphi\right|^{2}}{\partial
x_{1}}=\frac{\partial \left|{\rm d}\varphi\right|^{2}}{\partial x}\\
=2a_{31}(a_{12}^{2}e^{2z}+a_{11}^{2}e^{-2z}+a_{13}^{2})e^{2z\acute{}}
-2a_{31}(a_{22}^{2}e^{2z}+a_{21}^{2}e^{-2z}+a_{23}^{2})e^{-2z\acute{}},\\
\frac{\partial \left|{\rm d}\varphi\right|^{2}}{\partial
x_{2}}=\frac{\partial \left|{\rm d}\varphi\right|^{2}}{\partial y}\\
=2a_{32}(a_{12}^{2}e^{2z}+a_{11}^{2}e^{-2z}+a_{13}^{2})e^{2z\acute{}}
-2a_{32}(a_{22}^{2}e^{2z}+a_{21}^{2}e^{-2z}+a_{23}^{2})e^{-2z\acute{}},\\
\frac{\partial \left|{\rm d}\varphi\right|^{2}}{\partial
x_{3}}=\frac{\partial \left|{\rm d}\varphi\right|^{2}}{\partial z}\\
=2\{(a_{12}^{2}+a_{33}a_{12}^{2})e^{2z}+(a_{33}a_{11}^{2}-
a_{11}^{2})e^{-2z}+a_{33}a_{13}^{2}\}e^{2z\acute{}}\\
-2\{(a_{33}a_{22}^{2}-a_{22}^{2})e^{2z}+(a_{33}a_{21}^{2}+
a_{21}^{2})e^{-2z}+a_{33}a_{23}^{2}\}e^{-2z\acute{}}\\
+2(a_{32}^{2}e^{2z}-a_{31}^{2}e^{-2z}).\\
\end{array}
\end{equation}
By Corollary \ref{PL124} $\varphi$ is an $\infty$-harmonic if and
only if
\begin{equation}\label{MAP601}
g(\nabla\,\varphi^{i}, \nabla\left|{\rm
d}\varphi\right|^{2})=0,\;\;\;\;\;\;\; i=1, 2, 3,
\end{equation}
which is equivalent to
\begin{eqnarray}\notag
0=&&\{a_{i2}a_{32}a_{12}^{2}e^{4z}+a_{i1}a_{31}a_{11}^{2}e^{-4z}\\\notag
&&+(a_{i2}a_{32}a_{13}^{2}+a_{i3}a_{12}^{2}+a_{i3}a_{33}a_{12}^{2})e^{2z}\\\notag
&&+(a_{i1}a_{31}a_{13}^{2}-a_{i3}a_{11}^{2}+a_{i3}a_{33}a_{11}^{2})e^{-2z}\\\notag
&&+a_{i1}a_{31}a_{12}^{2}+a_{i2}a_{32}a_{11}^{2}+a_{i3}a_{33}a_{13}^{2}\}e^{2z\acute{}}\\\label{MAP993}
&&-\{a_{i2}a_{32}a_{22}^{2}e^{4z}+a_{i1}a_{31}a_{21}^{2}e^{-4z}\\\notag
&&+(a_{i2}a_{32}a_{23}^{2}-a_{i3}a_{22}^{2}+a_{i3}a_{33}a_{22}^{2})e^{2z}\\\notag
&&+(a_{i1}a_{31}a_{23}^{2}+a_{i3}a_{21}^{2}+a_{i3}a_{33}a_{21}^{2})e^{-2z}\\\notag
&&+a_{i1}a_{31}a_{22}^{2}+a_{i2}a_{32}a_{21}^{2}+a_{i3}a_{33}a_{23}^{2}\}e^{-2z\acute{}}\\\notag
&&+a_{i3}(a_{32}^{2}e^{2z}-a_{31}^{2}e^{-2z}),\;\;\;\;\;\;i=1,2,3.
\end{eqnarray}
Case (A): $a_{31}^{2}+a_{32}^{2}+a_{33}^{2}=0$. It follows that
$z\acute{}=a_{31}x+a_{32}y+a_{33}z=0$, and the Equation (\ref
{MAP993}) becomes
\begin{equation}\label{N31463}
\begin{array}{lll}
a_{i3}\{(a_{12}^{2}+a_{22}^{2})e^{2z}-(a_{11}^{2}+a_{21}^{2})e^{-2z}\}=0,\;\;\;i=1,2,3,
\end{array}
\end{equation}
which gives the solutions $a_{11}=a_{12}=a_{21}=a_{22}=0$, or, $
a_{i3}=0,\;for\; i=1,2,3.$ These give the linear $\infty$-harmonic
maps of the form (\ref{MAP8300}) and (\ref{MAP8200}).\\

Case (B): $a_{31}^{2}+a_{32}^{2}+a_{33}^{2}\neq 0$. We use
 Equation (\ref {MAP993}) and the fact that the functions
 $e^{k_1t},
 e^{-k_1t}; e^{k_2t},
 e^{-k_2t}; e^{k_3t},
 e^{-k_3t}; e^{k_4t},
 e^{-k_4t}; e^{k_5t},
 e^{-k_5t}$ with $k_1, \ldots, k_5$ distinctive are linearly
 independent to conclude that $\varphi$ is $\infty$-harmonic if and
 only if
\begin{equation}\label{N3301}
\left\{\begin{array}{rl} a_{i3}a_{32}^{2}=0,\;\;\;\;\;\;\;\;\;\;\;\langle 1\rangle\\
a_{i3}a_{31}^{2}=0,\;\;\;\;\;\;\;\;\;\;\;\langle 2\rangle\\
 a_{i2}a_{32}a_{12}^{2}=0,\;\;\;\;\;\;\;\;\;\;\;\langle3\rangle\\
a_{i1}a_{31}a_{11}^{2}=0,\;\;\;\;\;\;\;\;\;\;\;\langle 4\rangle\\
a_{i2}a_{32}a_{13}^{2}+a_{i3}a_{12}^{2}+a_{i3}a_{33}a_{12}^{2}=0,\;\;\;\;\;\;\;\;\;\;\;\langle5\rangle\\
a_{i1}a_{31}a_{13}^{2}-a_{i3}a_{11}^{2}+a_{i3}a_{33}a_{11}^{2}=0,\;\;\;\;\;\;\;\;\;\;\;\langle 6\rangle\\
a_{i1}a_{31}a_{12}^{2}+a_{i2}a_{32}a_{11}^{2}+a_{i3}a_{33}a_{13}^{2} =0,\;\;\;\;\;\;\;\;\;\;\;\langle7\rangle\\
  a_{i2}a_{32}a_{22}^{2}=0,\;\;\;\;\;\;\;\;\;\;\;\langle 8\rangle\\
a_{i1}a_{31}a_{21}^{2}=0,\;\;\;\;\;\;\;\;\;\;\;\langle 9\rangle\\
a_{i2}a_{32}a_{23}^{2}-a_{i3}a_{22}^{2}+a_{i3}a_{33}a_{22}^{2}=0,\;\;\;\;\;\;\;\;\;\;\;\langle 10\rangle\\
a_{i1}a_{31}a_{23}^{2}+a_{i3}a_{21}^{2}+a_{i3}a_{33}a_{21}^{2}=0,\;\;\;\;\;\;\;\;\;\;\;\langle 11\rangle\\
a_{i1}a_{31}a_{22}^{2}+a_{i2}a_{32}a_{21}^{2}+a_{i3}a_{33}a_{23}^{2} =0.\;\;\;\;\;\;\;\;\;\;\;\langle 12\rangle\\
\end{array}\right.
\end{equation}
It follows from $\langle 1\rangle$ and $\langle 2\rangle$ of
(\ref{N3301}) that
\begin{equation}\label{N33201}
 a_{i3}=0, \;\; {\rm or},\;a_{31}= a_{32}=0,\;\;\;i=1, 2, 3.
\end{equation}

Case (B$_1$): $a_{i3}=0,\;\;\;i=1, 2, 3$ and hence
$a_{31}^{2}+a_{32}^{2}\neq
0$.\\
Performing $\langle3\rangle+\langle4\rangle+\langle7\rangle,
\langle8\rangle+\langle9\rangle+\langle12\rangle$ separatively
yields
\begin{equation}\label{N31205}
\left\{\begin{array}{rl}
(a_{11}^{2}+a_{12}^{2})(a_{i1}a_{31}+a_{i2}a_{32})=0\\
(a_{21}^{2}+a_{22}^{2})(a_{i1}a_{31}+a_{i2}a_{32})=0,\\
\end{array}\right.\;\;\;\;i=1,2,3,
\end{equation}
which gives us solutions $ a_{11}=a_{21}=a_{12}=a_{22}=0$. These
give $\infty$-harmonic linear automorphisms of the form defined in
(\ref{MAP8100}).\\
Case (B$_2$): $a_{31}=a_{32}=0$, and hence $a_{33}\neq 0$.\\
In this case, Equation (\ref{N3301}) reduces to
\begin{equation}\label{N31230}
 \left\{\begin{array}{rl}
a_{i3}a_{12}^{2}(1+a_{33})=0\\
a_{i3}a_{11}^{2}(a_{33}-1)=0\\
a_{i3}a_{33}a_{13}^{2} =0,\\
a_{i3}a_{22}^{2}(a_{33}-1)=0\\
a_{i3}a_{21}^{2}(1+a_{33})=0\\
a_{i3}a_{33}a_{23}^{2} =0,\\
\end{array}\right.\;\;\;\;i=1,2,3.
\end{equation}
We solve this system by considering the following three case:\\ (I)
$a_{33}=1$. By (\ref{N31230}), we have
\begin{equation}\label{N31246}
 \left\{\begin{array}{rl}
 a_{i3}a_{12}^{2}=0\\
a_{i3}a_{13}^{2} =0\\
a_{i3}a_{21}^{2}=0\\
a_{i3}a_{23}^{2} =0,\\
\end{array}\right.\;\;\;\;i=1,2,3.
\end{equation}
Letting $i=3$ we conclude that $a_{12}=a_{13}=a_{21}=a_{23}=0$, which give the solutions of the form (\ref{MAP8500}). \\
(II) $a_{33}=-1$. In this case, (\ref{N31230}) reduces to
\begin{equation}\label{N31257}
 \left\{\begin{array}{rl}
a_{i3}a_{11}^{2}=0\\
a_{i3}a_{13}^{2} =0\\
a_{i3}a_{22}^{2}=0\\
a_{i3}a_{23}^{2} =0,\\
\end{array}\right.\;\;\;\;i=1,2,3.
\end{equation}
Letting  $i=3$ we conclude that $a_{11}=a_{22}=a_{13}=a_{23}=0$, which give the solutions of the form (\ref{MAP8600}).\\
(III) $a_{33}\neq \pm 1,\;0$. Then, (\ref{N31230}) becomes
\begin{equation}\label{N31268}
 \left\{\begin{array}{rl}
a_{i3}a_{12}^{2}=0\\
a_{i3}a_{11}^{2}=0\\
a_{i3}a_{13}^{2} =0\\
a_{i3}a_{22}^{2}=0\\
a_{i3}a_{21}^{2}=0\\
a_{i3}a_{23}^{2} =0,\\
\end{array}\right.\;\;\;\;i=1,2,3.
\end{equation}
Letting $i=3$ we get $a_{11}=a_{12}=a_{13}=a_{21}=a_{22}=a_{23}=0$, which give the solutions of the form (\ref{MAP8400}). \\
Summarizing all results in the above cases we obtain the Theorem.
\end{proof}
\begin{corollary}
Every element of the subgroup
\begin{eqnarray}
\left\{ \varphi \in {\rm GL}(\mathbb{R}^3):\;\;
 \varphi(X)=\left(\begin{array}{ccc}
\lambda & 0 & 0\\
0 & \mu & 0\\
0& 0 & 1
\end{array}\right)\left(\begin{array}{ccc}
x\\y\\z\end{array}\right), \lambda \mu\ne 0\right\}
\end{eqnarray}
 of the linear automorphism group of Sol space is $\infty$-harmonic.
\end{corollary}
\begin{proof}
It follows from Theorem \ref{KL1534} that every element of the
subgroup is an $\infty$-harmonic map. A straightforward checking
shows that the inverse elements and the products of elements of the
subgroup are also $\infty$-harmonic.
\end{proof}
\begin{remark}
It follows from our theorem that the maximum rank of linear
$\infty$-harmonic endomorphisms of Sol space is $3$, so we can have
linear $\infty$-harmonic diffeomorphisms which have constant energy
density and which are not isometries. Using the energy density
formula (\ref{SOLS}) we can check that all rank $2$ linear
$\infty$-harmonic maps from Sol into itself have non-constant energy
density.
\end{remark}

{\bf Acknowledgments}\\

  I would like to thank my adviser Prof. Dr. Ye-Lin Ou for his guidance, help, and encouragement through
many invaluable discussions, suggestions, and stimulating questions
during the preparation of this work.

\end{document}